\documentclass[review]{elsarticle}
\usepackage{lineno}
\usepackage[latin1]{inputenc}
\usepackage{times}
\usepackage{amssymb,mathrsfs, amsmath}
\usepackage{amsmath,amsthm}
\usepackage{amssymb}
\usepackage{latexsym}
\usepackage{amsfonts}
\usepackage{epsfig}
\usepackage{graphicx}
\usepackage{graphics}
\usepackage[all]{xy}
\usepackage{tikz-cd}
\usepackage[mathscr]{euscript}
\usepackage{mathrsfs}
\usepackage{stmaryrd}
\usepackage{xcolor}

\newtheorem{thm}{Theorem}[section]

\newtheorem{defn}{Definition}[section]
\newtheorem{prop}{Proposition}[section]

\newtheorem{rem}{Remark}[section]

\newtheorem{exmpl}{Example}[section]

\journal{... }

\begin{document}
\begin{frontmatter}

\title{Cohomology and  deformation of   Leibniz superalgebras}

%%% Group authors per affiliation:
%%\author{\fnref{myfootnote}}
%%\address{}
%%\fntext[myfootnote]{This author is supported by CSIR,\textsc{India}.}
%\author{RB Yadav \fnref{myfootnote}\corref{mycorrespondingauthor}}
%\address{Indian Statistical Institute Tezpur, Assam 784028, \textsc{India}}
%\ead{rbyadav15@gmail.com}
%\author{Goutam Mukherjee\fnref{myfootnote}\corref{mycorrespondingauthor}}
%\address{Stat-Math Division,  Indian Statistical Institute, Kolkata, 700108, \textsc{India}}
%\ead{goutam@isical.ac.in}
%\fntext[myfootnote]{This author is supported by CSIR,\textsc{India}.}
%
%%
\author{RB Yadav\fnref{}\corref{mycorrespondingauthor}}
%\address{Sikkim University, Gangtok, Sikkim, 737102, \textsc{India}}
\cortext[mycorrespondingauthor]{Corresponding author}
\ead{rbyadav15@gmail.com, rbyadav01@cus.ac.in}
\author{Arpan Sharma\corref{mycoauthor}}
\address{Sikkim University, Gangtok, Sikkim, 737102, \textsc{India}}
\ead{arpansachin@gmail.com}
%\author{ Rinkila Bhutia\corref{mycoauthor}}

%\ead{rbhutia@cus.ac.in}

\begin{abstract}
In this article, we give a characterization of Leibniz superalgebras. We study   cohomology theory  of Leibniz superalgebras.      Using the cohomology we study the  formal deformation theory of Leibniz superalgebras. We do some cohomology computation in lower dimensions. We establish a one to one correspondence between extensions of a Leibniz superalgebra  and its cohomology group in dimension $2$.

\end{abstract}

\begin{keyword}
\texttt{Leibniz superalgebra, cohomology, Extension,  formal deformations, infinitesimal deformation }
\MSC[2020] 17A70 \sep 17B99 \sep 16S80 \sep 13D10 \sep 13D03  \sep 16E40
\end{keyword}

\end{frontmatter}

%\linenumbers

%\maketitle

\section{Introduction}\label{rbsec1}
Leibniz algebras were introduced by J.L Loday in \cite{MR1252069} as a noncommutative generalization of Lie algebras.   Lie superalgebras were studied and a classification was given by Kac \cite{MR486011}.  Leits \cite{MR0422372} introduced a cohomology for Lie superalgebras.  Leibniz superalgebras  are a noncommutative generalizations of Lie superalgebras. Leibniz superalgebras were studied in \cite{MR2286721}, \cite{MR2102079}.

The deformation is a tool to study a mathematical object by deforming it into a family of the same kind of objects depending on a certain parameter. Algebraic deformation theory was introduced by Gerstenhaber for rings and algebras \cite{MR171807},\cite{MR0207793},\cite{MR240167}, \cite{MR389978}, \cite{MR704600}. Deformation theory of Lie superalgebras was introduced and studied by Binegar \cite{MR871615}.  Recently, algebraic deformation theory has been studied by several authors \cite{MR4156100}, \cite{MR4171679}, \cite{MR4120091} etc.

  Cohomology of colour Leibniz algebras was defined in \cite{MR1888379}. By considering Leibniz superalgebras as a particlur  colour Leibniz algebra one can  define  cohomology of Leibniz superalgebras. But in \cite{MR1888379} author does not do any  cohomology computations  of Leibniz superalgebras. 

Purpose of this paper is to  discuss   deformation cohomology, do some cohomology computations in lower dimensions  and introduce formal deformation theory of Leibniz superalgebras.  Organization of the paper is as follows. In Section \ref{rbsec2}, we recall definition of Leibniz superalgebra and  some  examples. In this section we introduce an example of Leibniz superalgebra \ref{LSE101}, \ref{rbLSe1}.  In fact, Example \ref{LSE101} provides a class of examples.  In Section \ref{rbsec3}, we give a characterization of Leibniz superalgebras and  study  cohomology of Leibniz superalgebras. In Section \ref{rbsec4}, we compute cohomology of Leibniz superalgebras in degree $0$ and dimension $0$, $1$ and $2$.  In Section \ref{rbsec5}, we introduce  deformation theory of  Leibniz superalgebras. In this section  we see  that infinitesimals of deformations are  cocycles . Also, in this section we give an example of a formal deformation of a Leibniz superalgebras. In Section \ref{rbsec6}, we study equivalence of two formal deformations and prove that infinitesimals of any two equivalent deformations are cohomologous.

\section{Leibniz Superalgebras}\label{rbsec2}
In this section, we recall definitions of Leibniz superalgebra and  module over a Leibniz superalgebras. We give some examples of Leibniz superalgebras.  Throughout the paper we denote a fixed field  by $K$.  Also, we denote the ring of formal power series with coefficients in $K$ by $K[[t]]$.
In any $\mathbb{Z}_2$-graded vector space $V$  we use a notation in which we replace degree $deg(a)$ of an element $a\in V$ by `a' whenever $deg(a)$ appears in an exponent; thus, for example $(-1)^{ab}=(-1)^{deg(a)deg(b)}$.
\begin{defn}
  Let $V=V_0\oplus V_1$ and $W=W_0\oplus W_1$ be $\mathbb{Z}_2$-graded vector spaces over a field $K$. A linear map $f:V \to W$ is said to  be homogeneous of degree $\alpha$ if $\deg(f(a))- \deg(a)=\alpha$, for all $a\in V_\beta$, $\beta\in \{0,1\}$. We write $(-1)^{\deg(f)}=(-1)^f$. Elements of $ V_\beta$ are called homogeneous of degree $\beta.$
\end{defn}
\begin{defn}\cite{MR2286721} A  (left) Leibniz superalgebra is defined as a $\mathbb{Z}_2$-graded vector space $L=L_0\oplus L_1$ over a field $K$ equipped with a bilinear map $[-,-]:L\times L\to L$ satisfying the following conditions:
\begin{enumerate}
  \item $[a,b]\in L_{\alpha+\beta}$,
  \item $[[a,b],c]=[a,[b,c]]-(-1)^{\alpha\beta}[b,[a,c]]$, \hspace{3cm}(Leibniz identity)
\end{enumerate}
for all $a\in L_{\alpha}$, $b\in L_{\beta}$ and $c\in L_{\gamma}$. If second condition in the Definition \ref{rbsec2} is replaced by $[x,[y,z]]=[[x,y],z]-(-1)^{yz}[[x,z],y]$, then $L$ is called right Leibniz superalgebra. In this paper we consider only left  Leibniz superalgebra.  Let $L_1$ and $L_2$ be two Leibniz superalgebras. A homomorphism $f:L_1\to L_2$ is  a $K$-linear map such that $f([a,b])=[f(a),f(b)].$  Given a Leibniz superalgebra $L$ we denote by $[L,L]$ the vector subspace of L spanned  by the set $\{[x,y]: x,y\in L\}$. A Leibniz superalgebra $L$ is called abelian if $[L,L]=0.$
\end{defn}
\begin{exmpl}
  Clearly, every Lie superalgebra is a Leibniz superalgebra. A Leibniz superalgebra $L=L_0\oplus L_1$ is a Lie superalgebra if $[a,b]=-(-1)^{ab}[b,a]$ for all $a\in L_{\alpha}$, $b\in L_{\beta}$.
\end{exmpl}
\begin{exmpl}
  Given any $\mathbb{Z}_2$-graded vector space $V=V_0\oplus V_1$ we can define  a multiplication on $V$ by $[x,y]=0$, for all $x\in V_\alpha$, $y\in V_\beta$. This gives an abelian Leibniz superalgebra structure on $V.$
\end{exmpl}
\begin{exmpl}\label{LSE101}
  Let $A=A_0\oplus A_1$ be an associative $K$-superalgebra equipped with a homogeneous linear map $T:A\to A$ of degree $0$ and satisfying the condition
  \begin{equation}\label{rbsl1}
     T(a(Tb))=(Ta)(Tb)=T((Ta)b),
  \end{equation}
   for all $a,b\in A$. We  define a bilinear map on A by $$[a,b]=(Ta)b-(-1)^{ab}b(Ta),$$ for all $a\in A_{\alpha}$, $b\in A_{\beta}$. One can easily verify that $[-,-]$  satisfies the two conditions for a Leibniz superalgebra. This makes A a Leibniz superalgebra that we denote by $A_{SL}$. If $T=Id,$ then $A_{SL}$ turns out to be a Lie superalgebra. If $T$ is an algebra map on $A$ which is idempotent ($T^2=T$), then condition \ref{rbsl1} is satisfied. If $T$ is a square-zero derivation, that is, $T(ab)=(Ta)b+a(Tb)$ and $T^2=0,$ then the condition \ref{rbsl1} is satisfied.
\end{exmpl}
\begin{exmpl}
 Let V be a $\mathbb{Z}_2$-graded $K$-vector space.  The free Leibniz superalgebra $\mathcal{L}$(V ) is the universal Leibniz superalgebra for maps from $V$ to Leibniz superalgebras.  Let $ \overline{T}(V ) := \oplus_{n\ge 1}V^{\otimes n}$ be the reduced tensor module. \cite{MR2286721}  $\overline{T}(V )$ is the free Leibniz superalgebra over $V$ with the multiplication defined inductively by
 \begin{enumerate}
   \item $[v,x]=v\otimes x,$ for all $x\in \overline{T}(V )$, $v\in V$
   \item $[y\otimes v, x]=[y,v\otimes x]-(-1)^{yv}v\otimes[y,x]$, for all $x\in \overline{T}(V )$, $v\in V$ and homogeneous $y\in \overline{T}(V )$.
 \end{enumerate}

\end{exmpl}
\begin{exmpl}\label{rbLSe1}
  Let $L=L_0\oplus L_1$ be a $\mathbb{Z}_2$-graded $K$-vector space, where $L_0$ is two dimensional subspace of $L$ generated by $\{x,y\}$ and  $L_1$ is 1-dimensional generated by $\{z\}$. Define a bilinear map $[-,-]:L\times L\to L$ given by $$[y,x]=x,\;[y,y]=x,\;[x,x]=[x,y]=[x,z]=[z,x]=[z,z]=[z,y]=[y,z]=0.$$ One can easily verify that $L$ together with $[-,-]$ is a Leibniz superalgebra.
\end{exmpl}
\begin{defn}\cite{MR2286721}
  Let $L=L_0\oplus L_1$ be a Leibniz superalgebra with Leibniz superalgebra bracket denoted by $\pi$. A $\mathbb{Z}_2$-graded vector space $M=M_0\oplus M_1$ over the field K is called a module over L if there exist two linear maps (left) $\rho^L:L\otimes M\to M$ and (right) $\rho^R:M\otimes L\to M$ such that following conditions are satisfied:
  
  \begin{align}
        \rho^L(\pi(a\otimes b)\otimes m)&= \rho^L(a\otimes \rho^L(b\otimes m)-(-1)^{ab}\rho^L(b\otimes \rho^L(a\otimes m) \label{rep11}\\
        \rho^R( \rho^L(a\otimes m)\otimes b)&=\rho^L(a\otimes \rho^R(m\otimes b))-(-1)^{am}\rho^R(m\otimes \pi(a\otimes b)) \label{rep22}\\
        \rho^R(\rho^R(m\otimes a)\otimes b)&=\rho^R(m\otimes \pi(a\otimes b))-(-1)^{ma}\rho^L(a\otimes \rho^R(m\otimes b)) \label{rep33}
    \end{align}
  
  If we write $\rho^L(a\otimes m)=[a,m]$ and $\rho^R(m\otimes a)=[m,a]$  above three  conditions can be written as:
%\begin{enumerate}
\begin{align}
  [[a,b],m]=[a,[b,m]]-(-1)^{ab}[b,[a,m]] \label{rep1}\\
  [[a,m],b]=[a,[m,b]]-(-1)^{am}[m,[a,b]] \label{rep2}\\
  [[m,a],b]=[m,[a,b]]-(-1)^{ma}[a,[m,b]] \label{rep3}
\end{align}
%\end{enumerate}
for all $a\in L_\alpha$, $b\in L_\beta,$ $m\in M_\gamma$, $\alpha,\beta,\gamma\in\{0, 1\}$.
\end{defn}
Clearly, every Leibniz superalgebra is a module over itself. 
\section{Characterization of Leibniz Superalgebras} \label{rbsec3}

Let $V=V_0\oplus V_1$ be a   $\mathbb{Z}_2$-graded vector space.
 Consider the permutation group $S_n.$ For any $X=(X_1,\ldots, X_n)$ with $X_i\in V_{x_i}$ and $\sigma\in S_n$, define 
$$ K(\sigma, X)= |\{ (i,j): i<j,\; X_{\sigma(i)}\in V_1, X_{\sigma(j)}\in V_1,  \; \sigma(j) < \sigma(i)\} |$$
and $\epsilon(\sigma, X)=\epsilon(\sigma)(-1)^{K(\sigma,X)},$ 
where $| A |$ denotes cardinality of a set $A, $ $\epsilon(\sigma)$ is the signature of $\sigma.$\\
Let $V$ and $W$ be  $\mathbb{Z}_2$-graded  vector spaces.  For each  $n\in \mathbb{N},$ define   $\mathcal{F}_{n,\alpha}(V,W)$ as  the vector space of all homogeneous $n$-linear mappings  $f:V \underset{n\; times}{\underbrace{\times\cdots\times}}V\to W$ of degree $\alpha.$ Denote the set of all $n$-linear mappings  $f:V \underset{n\; times}{\underbrace{\times\cdots\times}}V\to W$ by  $\mathcal{F}_{n}(V,W)$.  Clearly, we have   $\mathcal{F}_{n}(V,W)=\mathcal{F}_{n,0}(V,W)\oplus \mathcal{F}_{n, 1}(V,W)$. Write  $\mathcal{F}_{0}(V,W)=W$ and  $\mathcal{F}_{-n}(V,W)=0$,  $\forall n\in \mathbb{N}.$  Take $\mathcal{F}(V,W)=\bigoplus_{n\in \mathbb{Z}}\mathcal{F}_{n}(V,W)$. 
We know that $\mathcal{F}_n(V,W)\cong  Hom(V^{\otimes n}, W)$. We write  $\mathbf{E}_n= Hom(V^{\otimes n+1}, V)$ and  $\mathbf{E}=\bigoplus_{n\in \mathbb{Z}}\mathbf{E}_n$. If we denote by $\mathbf{E}_{n,\alpha}=Hom_{\alpha}(V^{\otimes n +1},V)$ as the vector space of all homogeneous elements of $\mathbf{E}_n$ of  degree $\alpha\in \mathbb{Z}_2$, then $\mathbf{E}_n=\mathbf{E}_{n,0}\oplus \mathbf{E}_{n,1}$ and $\mathbf{E}$ is a $\mathbb{Z}\times\mathbb{Z}_2$-graded vector space. Here, we define  $ f\in Hom(V^{\otimes n +1},V)$ to have   degree $(n,f) \in \mathbb{Z}\times\mathbb{Z}_2$ where, $f=deg(f(x_1,\ldots,x_{n+1}))-deg((x_1,\ldots,x_{n+1}))$ taken modulo $\mathbb{Z}_2$. Define a product $\circ$ on $\mathbf{E}$ as follows:\\
For every $F \in \mathbf{E}_{n,f}$ and $F' \in \mathbf{E}_{n',f'}$
\[
F \circ F'=\sum_{k=1}^{n+1}(-1)^{(k-1)n'}F \circ_k F'
\]
and $\circ_k$ is defined by
\begin{multline*}
F \circ_kF'(x_1,\ldots,x_{n+n'+1}) \\ =\sum_{\sigma \in S_{(k-1,n')}}\epsilon(\sigma,z_k)(-1)^{f'(x_{\sigma(1)}+ \cdots + x_{\sigma(k-1)})} F(x_{\sigma(1)}, \ldots , x_{\sigma(k-1)},\\
F'(x_{\sigma(k)},\ldots,x_{\sigma(k+n'-1)},x_{k+n'}),x_{k+n'+1},\ldots,x_{n+n'+1}),
\end{multline*}
where $z_k=(x_1,\ldots,x_{k+n'-1})$,  $x_1,\ldots,x_{n+n'+1} $ are homogeneous  elements in $V$ %$z_k=(x_1,\ldots,x_{k+n'-1})$
and $S_{(k-1,n')}$ consists of permutations $\sigma\in S_{(k-1+n')}$ such that  $\sigma(1)<\cdots<\sigma(k-1)$ and  $\sigma(k)<\cdots<\sigma(k-1+n').$

Using \cite{Balavoine1997}, we have following theorem:
\begin{thm}\label{MLT3}
$\mathbf{E}$ is a $\mathbb{Z}\times \mathbb{Z}_2$-graded Lie algebra
  with the bracket  $[-,-]_B$ defined  by 
  \begin{equation}\label{GNR}
      [F,F']_B=F\circ F'-(-1)^{nn'+ff'}F'\circ F
  \end{equation}
  for $F\in \mathbf{E}_{n,f}$,  $F'\in \mathbf{E}_{n',f'}$. 
\end{thm}
Let $W_1$ and $W_2$ be two $\mathbb{Z}_2$ graded vector spaces and elements in $W_1$ will be denoted by $x; y; x_i$ and elements in $W_2$ will be denoted by $u; v; v_i$. Let $g:W_2^{\otimes n} \to W_1$ be a linear map. We can construct a linear map $\hat{g} \in Hom((W_1 \oplus W_2)^{\otimes n},W_1 \oplus W_2)$ by
\[
\hat{g}((x_1+v_1) \otimes \cdots \otimes (x_n+v_n))=g(v_1 \otimes \cdots \otimes v_n)
\]
We have $$(W_1 \oplus W_2)^{\otimes n}=\oplus_{\substack{i_j=1, 2\\j=1,\cdots,n}}   W_{i_1}\otimes \cdots \otimes W_{i_n}.$$ In general, for a given linear map $f: W_{i_1}\otimes \cdots \otimes W_{i_n} \to W_j$, where $i_1, \ldots, i_n, j \in \{1,2\}$, we define a linear map $\hat{f} \in Hom((W_1 \oplus W_2)^{\otimes n},W_1 \oplus W_2)$ by
\begin{equation}\label{extension}
\hat{f} = \begin{cases}
   f & on \,\,\,W_{i_1}\otimes \cdots \otimes W_{i_n}\\
   0 & elsewhere
\end{cases}
\end{equation}
For simplicity of the notation we shall denote $\hat{f}$ by $f$ itself and recognize it from context.
\begin{defn}\label{RAdf30}
We define  a graded derivation of degree $(\alpha,\beta)$ of a $\mathbb{Z}\times \mathbb{Z}_2$-graded Lie algebra $E$ as  a  linear map $\partial :E\to E$ of degree $(\alpha,\beta)$ such that $\partial ([x,y])=[\partial x,y]+(-1)^{\alpha l +\beta m}[x,\partial y]$, for every homogeneous  $x$ of degree $(l,m)$.
A  differential $\mathbb{Z}\times \mathbb{Z}_2$-graded Lie algebra   $(E, [-,-],\partial)$ consists of  a $\mathbb{Z}\times \mathbb{Z}_2$-graded  Lie algebra   $(E, [-,-])$ and a graded derivation $\partial :E\to E$ of degree $ (1,0)$ such that $\partial^2=0$.   Maurer-Cartan elements of $(E, [-,-],\partial)$ is defined as an element $\omega$ of $E$  of degree $(1,0)$ such that $\partial\omega+\frac{1}{2}[\omega,\omega]=0.$ Every $\mathbb{Z}\times \mathbb{Z}_2$-graded Lie algebra $(E,[-,-])$ can be seen as   differential graded Lie algebra with $\partial=0$.
  Hence Maurer-Cartan elements of the $\mathbb{Z}\times \mathbb{Z}_2$-graded  Lie algebra   $E$ is an element $\omega$ of $E$  of degree $(1,0)$ such that $[\omega,\omega]=0.$
\end{defn}

\begin{thm}
    Let $L$ be $\mathbb{Z}_2$-graded vector space. Then, degree $0$ linear map $\pi:L \otimes L \to L$ is a Leibniz superbracket iff $[\pi,\pi]_B=0$, i.e. $\pi$ is a Maurer-Cartan element of the graded Lie algebra $\bigoplus_{n} Hom(L^{\otimes n},L)$ with the bracket $[-,-]_B$ defined in $\eqref{MLT3}$.
\end{thm}
    \begin{proof}
     By using direct definition we have $[\pi,\pi]_B=2 \pi \circ \pi$ and 
    \begin{align*}
        &\pi \circ \pi(x_1,x_2,x_3)\\
        = &\sum \limits _{k=1}^{2}(-1)^{(k-1)}\pi \circ_k \pi(x_1,x_2,x_3)\\
        = &\pi \circ_1 \pi(x_1,x_2,x_3) - \pi \circ_2 \pi(x_1,x_2,x_3)\\
        = &\sum \limits_{\sigma \in S_{(0,1)}}\epsilon(\sigma,(x_1))\pi(\pi(x_{\sigma(1)},x_2),x_3) \\
       & - \sum \limits_{\sigma \in S_{(1,1)}}\epsilon(\sigma,(x_1,x_2))(-1)^{\pi(x_{\sigma(1)})}\pi(x_{\sigma(1)},\pi(x_{\sigma(2)},x_3))\\
      = & \pi(\pi(x_1,x_2),x_3)-\pi(x_1,\pi(x_2,x_3))+(-1)^{x_1x_2}\pi(x_2,\pi(x_1,x_3))
      \end{align*}
      Hence, $[\pi,\pi]_B=0$ iff $\pi$ is a Leibniz superbracket.
    \end{proof}
\begin{thm}
   Let $(L,[-,-]_L)$ be a Leibniz superalgebra. Then  $M=M_0\oplus M_1$ is a module over $L$ with left and right actions $\rho^L:L \otimes M \to M$ and $\rho^R: M \otimes L \to M$ if and only if  $\Theta=\pi + \rho^L+\rho^R$  ($\pi$ is the Leibniz superbracket $[-,-]_L$) is a Maurer-Cartan element of the graded Lie algebra $(\bigoplus_{n} Hom((L \oplus M)^{\otimes n},L \oplus M)$ with the bracket $[-,-]_B$ defined in Theorem $\ref{MLT3}$.
\end{thm}
\begin{proof} By \eqref{extension}, $ 
\Theta((x+u)\otimes (y+v))=[x,y]_{L}+(\rho^L(x,v)+\rho^R(u,y)).$ 
  We have   $[\Theta,\Theta]=2 \Theta \circ \Theta$ and for  $x\in L_\alpha$, $y\in L_\beta$ and $w\in M_\gamma$, $\alpha,\beta,\gamma\in \{1,2\}$,
    \begin{align*}
    &\Theta \circ \Theta(x\otimes y\otimes w)\\
    &=\Theta \circ_1 \Theta(x\otimes y\otimes w)-\Theta \circ_2 \Theta(x\otimes y\otimes w)\\
    &=\Theta(\Theta(x\otimes y)\otimes w)-\Theta(x\otimes \Theta(y\otimes w))+(-1)^{\alpha\beta}
    \Theta(y\otimes \Theta(x\otimes w)\\
    & =\rho^L(\pi(x\otimes y)\otimes w)-\rho^L(x\otimes\rho^L(y\otimes w))+(-1)^{\alpha\beta}
    \rho^L(y\otimes \rho_L(x\otimes w)
    \end{align*}
     \begin{align*}
    &\Theta \circ \Theta(x\otimes w\otimes y)\\
    &=\Theta \circ_1 \Theta(x\otimes w\otimes y)-\Theta \circ_2 \Theta(x\otimes w\otimes y)\\
    &=\Theta(\Theta(x\otimes w)\otimes y)-\Theta(x\otimes \Theta(w\otimes y))+(-1)^{\alpha\gamma}
    \Theta(w\otimes \Theta(x\otimes y)\\
    & =\rho^R(\rho^L(x\otimes w)\otimes y)-\rho^L(x\otimes\rho_R(w\otimes y))+(-1)^{\alpha\gamma}
    \rho^R(w\otimes \pi(x\otimes y)
    \end{align*}
    \begin{align*}
    &\Theta \circ \Theta(w\otimes x\otimes y)\\
     &=\Theta \circ_1 \Theta(w\otimes x\otimes y)-\Theta \circ_2 \Theta(w\otimes x\otimes y)\\
    &=\Theta(\Theta(w\otimes x)\otimes y)-\Theta(w\otimes \Theta(x\otimes y))+(-1)^{\alpha\gamma}
    \Theta(x\otimes \Theta(w\otimes y)\\
    & =\rho^R(\rho^R(w\otimes x)\otimes y)-\rho^R(w\otimes\pi(x\otimes y))+(-1)^{\alpha\gamma}
    \rho^L(x\otimes \rho^R(w\otimes y)
    \end{align*}
     Note that $\Theta \circ \Theta (p\otimes q\otimes r)=0$ if two or more of $p, q, r $ are homogeneous elements of $M.$ Hence we conclude the result.
\end{proof}
\noindent
We define a linear map $d$ of degree $(1,0)$ on $\bigoplus_{n} Hom(L^{\otimes n},M)$ by $$d=\partial_{\pi+\rho^L+\rho^R}=[\pi+\rho^L+\rho^R,-]_B.$$ 
Let $A$ and $B$ be two vector spaces over the same field. Because there is one to one correspondence between $Hom(A^{\otimes n},B)$  and the space of all multilinear maps from $A^n$ to $B$, for simplicity and convenience we shall use the notations $f(x_1,\ldots,x_{n+1})$ and $f(x_1\otimes \ldots \otimes x_{n+1})$ for the same thing. No confusion should arise in such situations.
If we write $d=\oplus_{n} d^n$, then,   for $f \in Hom(L^{\otimes n},M)$, 
$d^nf=[\pi+\rho^L+\rho^R, f]$.
For $n\ge 1$, we write $C^{n}(L;M)=Hom(L^{\otimes n},M)$, $C^0(L;M)=M$ and $C^*(L;M)=\bigoplus_{n} Hom(L^{\otimes n}, M)$.
By using the fact that $\pi + \rho^L+\rho^R$ is a Maurer-Cartan element of the graded Lie algebra $\Big( \bigoplus_{n} Hom((L \oplus M)^{\otimes n},L \oplus M),[-,-]_B \Big)$, it can be easily shown that 
\begin{equation}\label{Leib9999}
   d\circ d=0. 
\end{equation}
We define $\delta^n: C^{n}(L;M)\to C^{n+1}(L;M)$ by $\delta^n=(-1)^{n+1}d^n$. 
We have 
\begin{eqnarray}
   && d^nf(x_1,\ldots,x_{n+1})\notag \\
   &&=\sum\limits_{i=1}^{n}(-1)^{i+1+x_i(x_1+\ldots + x_{i-1}+\alpha_f)}\rho^L(x_i,f(x_1,\dots,\hat{x}_i,\ldots,x_{n+1}))\notag\\
  && +\sum\limits_{k=1}^{n}\sum_{i=1}^{k} (-1)^{i+x_i(x_{i+1} + \cdots + x_k)}f(x_1,\ldots,\hat{x}_i,\ldots,x_k,\pi(x_i,x_{k+1}),x_{k+2},\ldots,x_{n+1})\notag\\
&&+(-1)^{n-1}\rho^R(f(x_1,\ldots,x_{n}),x_{n+1}),
\end{eqnarray}
for homogeneous  $f \in Hom(L^{\otimes n},M)$. Now if we define $\delta:\oplus_n C^{n}(L;M)\to \oplus_n C^{n}(L;M) $ by $\delta=\oplus_n \delta^n$, then using Equation \ref{Leib9999}, we have $\delta\circ\delta=0.$ 
Hence we have following theorem: 
\begin{thm}
    $\delta\circ \delta=0$, that is, $(C^{\ast}(L;M),\delta)$ is a cochain complex.
  \end{thm}
 
We denote $\ker(\delta^n)$ by $Z^n(L;M)$ and image of $(\delta^{n-1})$ by $B^n(L;M)$.
We call  the $n$-th  cohomology $Z^n(L;M)/B^n(L;M)$ of the cochain complex $\{C^{n}(L;M),\delta^n\}$ as the  $n$-th  cohomology of $L$ with coefficients in $M$ and  denote it by $H^{n}(L;M)$. Since $L$ is a module over itself. So we can consider  cohomology groups $H^{n}(L;L)$. We call $H^{n}(L;L)$ as the $n$-th deformation cohomology group of $L$.
We  have $$Z^n(L;M)=Z_0^n(L;M)\oplus Z_1^n(L;M),  \;B^n(L;M)=B_0^n(L;M)\oplus B_1^n(L;M),$$ where $Z_i^n(L;M)$and $B_i^n(L;M)$ are submodules of  $C_i^n(L;M)$, $i=0,1$. Since boundary map  $\delta^n:C^{n}(L;M)\to C^{n+1}(L;M)$ is homogeneous of degree $0$, we conclude that $H^{n}(L;M)$ is $\mathbb{Z}_2$-graded and $ H^{n}(L;M)=H_0^{n}(L;M)\oplus H_0^{n}(L;M),$  where $H_i^{n}(L;M)=Z_i^n(L;M)/B_i^n(L;M)$, $i=0,1$.
\begin{rem}
One can easily verify that cohomology defined by us in this section is cohomology of colour Leibniz algebras defined by ASKAR DZHUMADIL'DAEV  in \cite{MR1888379} when Leibniz superalgebra  is considered as  a particular colour Leibniz algebra. 
\end{rem}

\section{Cohomology of Leibniz Superalgebras in Low Degrees}\label{rbsec4}
Let $L=L_0\oplus L_1$ be a Leibniz superalgebra and $M=M_0\oplus M_1$ be a module over $L.$ For $m\in M_0=C_0^0(L;M)$, $f\in C_0^1(L;M)$ and $g\in C_0^2(L;M)$
\begin{equation}\label{SLCOH1}
  \delta^0 m(x)=-[m,x],
\end{equation}
\begin{equation}\label{SLCOH2}
 \delta^1 f(x_1,x_2)=-f([x_1,x_2])+[x_1,f(x_2)]+[f(x_1),x_2],
\end{equation}
\begin{eqnarray}\label{SLCOH3}
   \delta^2 g(x_1,x_2,x_3)&=& -g([x_1,x_2],x_3)-(-1)^{x_1x_2}g(x_2,[x_1,x_3])+g(x_1,[x_2,x_3])\nonumber\\
   &&+(-1)^{x_1g}[x_1,g(x_2,x_3)]-(-1)^{x_2x_1+x_2g}[x_2,g(x_1,x_3)]\nonumber\\
   &&-[g(x_1,x_2),x_3].
\end{eqnarray}
 The set $\{m\in M_0|[m,x]=0, \forall x\in L\}$ is called annihilator of $L$ in $M_0$ and is denoted by $ann_{M_0}L$.  We have \begin{eqnarray*}
          H_0^0(L;M) = \{m\in M_0|-[m,x]=0,\;\text{for all}\; x\in L\} = ann_{M_0}L.
         \end{eqnarray*}
 A homogeneous linear map $f:L\to M\;$ is called  derivation from $L\;$ to $M\;$ if $\delta^1f=0.$  For every $m\in M_0$ the map $x\mapsto [m,x]$ is called an inner derivation from  $L$ to $M$. We denote the vector spaces of derivations and inner derivations from  $L$ to $M$ by $Der(L;M)$ and $Der_{Inn}(L;M)$ respectively. By using \ref{SLCOH1}, \ref{SLCOH2} we have $$H_0^1(L;M)=Der(L;M)/Der_{Inn}(L;M).$$

Let $L$ be a Leibniz superalgebra and $M$ be a module over $L.$ We regard $M$ as an abelian Leibniz superalgebra.   An extension of $L$ by $M$ is an exact sequence
  \[\xymatrix{0\ar[r]& M\ar[r]^i &\mathcal{E}\ar[r]^\pi &L\ar[r] &0 }\tag{*}\]
 of Leibniz superalgebras  such that $ [x,i(m)]=[\pi(x),m],\; [i(m),x]=[m,\pi(x)].$ 
  %We say that an extension $L$ by $M$ \[\xymatrix{0\ar[r]& M\ar[r]^i &\mathcal{E}\ar[r]^\pi &L\ar[r] &0 }\] is inessential if there exists a Leibniz superalgebra homomorphisms $u:L\to \mathcal{E}$ such that $\pi \circ u=Id_L.$
 The exact sequence $(*)$ regarded  as a sequence of $K$-vector spaces, splits. Therefore without any loss of generality we may assume that  $\mathcal{E}$ as a  $K$-vector space coincides with the direct sum $L\oplus M$ and that $i(m)=(0,m),$ $\pi(x,m)=x.$ Thus we have  $\mathcal{E}=\mathcal{E}_0\oplus \mathcal{E}_1,$ where $\mathcal{E}_0=L_0\oplus M_0$, $\mathcal{E}_1=L_1\oplus M_1.$ The multiplication in $\mathcal{E}=L\oplus M$ has then necessarily the form $$ [(0,m_1),(0,m_2)]   =  0,\;  [(x_1,0),(0,m_1)] = (0,[x_1,m_1]),$$
 $$[(0,m_2),(x_2,0)] = (0,[m_2,x_2]),\; [(x_1,0),(x_2,0)] = ([x_1,x_2],h(x_1,x_2)),$$  for some $h\in C_0^2(L;M)$, for all  homogeneous $x_1,x_2\in L$, $m_1,m_2\in M.$
Thus, in general, we have
 \begin{equation}\label{SLCOH7}
[(x,m),(y,n)]=([x,y],[x,n]+[m,y]+h(x,y)),
\end{equation}
for all  homogeneous $(x,m)$, $(y,n)$ in $\mathcal{E}=L\oplus M.$\\
 Conversely, let $h:L\times L\to M$ be a bilinear homogeneous map of degree $0$. For homogeneous $(x,m)$, $(y,n)$ in $\mathcal{E}$ we define multiplication in $\mathcal{E}=L\oplus M$ by Equation \ref{SLCOH7}.
 For homogeneous $(x,m)$, $(y,n)$ and $(z,p)$ in $\mathcal{E}$ we have
\begin{eqnarray}\label{SLCOH4}
  &&[[(x,m),(y,n)],(z,p)]\nonumber\\
  &=&([[x,y],z],[[x,y],p]+[[x,n],z]+[[m,y],z]+[h(x,y),z]+h([x,y],z))\nonumber\\
\end{eqnarray}
\begin{eqnarray}\label{SLCOH5}
&&[(x,m),[(y,n),(z,p)]]\nonumber\\
&=&([x,[y,z]], [x,[y,p]]+[x,[n,z]]+[m,[y,z]]+[x,h(y,z)]+h([x,y],z)\nonumber\\
\end{eqnarray}
\begin{eqnarray}\label{SLCOH6}
&&[(y,n),[(x,m),(z,p)]]\nonumber\\
&=&([y,[x,z]],[y,[x,p]]+[y,[m,z]]+[n,[x,z]]+[y,h(x,z)]+h(y,[x,z]))\nonumber\\
\end{eqnarray}
 From Equations \ref{SLCOH4}, \ref{SLCOH5}, \ref{SLCOH6} we conclude that $\mathcal{E}=L\oplus M$ is a Leibniz superalgebra  with product given by Equation \ref{SLCOH7}  if and only if $\delta^2 h=0.$
 We denote the Leibniz superalgebra given by Equation \ref{SLCOH7} using notation $\mathcal{E}_h$. Thus for every cocycle $h\in C_0^2(L;M)$ there exists an extension
 %\begin{equation}
 \label{SLCOH8}
   \[E_h:\xymatrix{0\ar[r]& M\ar[r]^i &\mathcal{E}_h\ar[r]^\pi &L\ar[r] &0 }\]
% \end{equation}
of $L$ by $M$, where $i$ and $\pi$ are inclusion and projection maps, that is, $i(m)=(0,m),$ $\pi(x,m)=x$.
 We say that two extensions
  \[\xymatrix{0\ar[r]& M\ar[r] &\mathcal{E}^i\ar[r] &L\ar[r] &0 } \;(i=1,2)\]
  of $L$ by $M$ are equivalent if there is a Leibniz superalgebra isomorphism $\psi:\mathcal{E}^1\to \mathcal{E}^2$ such that following diagram commutes:
\[
\xymatrix{
  0 \ar[r] & M \ar[d]_{Id_M} \ar[r]^-{} & \mathcal{E}^1 \ar[d]_-{\psi} \ar[r]^-{} & L \ar[d]^-{Id_L} \ar[r] & 0 \\
  0 \ar[r] & M \ar[r]_-{} & \mathcal{E}^2 \ar[r]_-{} & L \ar[r] & 0
}
\tag{**}\]
We use $F(L,M)$to denote the set of all equivalence classes of extensions of   $L$ by $M$. Equation \ref{SLCOH7} defines a mapping of $Z_0^2(L;M)$
 onto $F(L,M)$. If for $h,h'\in Z_0^2(L;M)$ $E_h$ is equivalent to $E_{h'}$, then commutativity of diagram $(**)$ is equivalent to $$ \psi(x,m)=(x,m+f(x)),$$  for some $f\in C_0^1(L;M)$.
 We have
 \begin{eqnarray}
   &&\psi([(x_1,m_1),(x_2,m_2)])\nonumber\\ &=& \psi([x_1,x_2],[x_1,m_2]+[m_1,x_2]+h(x_1,x_2))\nonumber \\
    &=& ([x_1,x_2],[x_1,m_2]+[m_1,x_2]+h(x_1,x_2)+f([x_1,x_2])),\nonumber\\
 \end{eqnarray}
  \begin{eqnarray}
    &&[\psi(x_1,m_1),\psi(x_2,m_2)] \nonumber\\&=& [(x_1,m_1+f(x_1)),(x_2,m_2+f(x_2))] \nonumber\\
     &=& ([x_1,x_2],[x_1,m_2+f(x_2)]+[m_1+f(x_1),x_2]+h'(x_1,x_2)).\nonumber\\
  \end{eqnarray}
 Since $\psi([(x_1,m_1),(x_2,m_2)])=[\psi(x_1,m_1),\psi(x_2,m_2)] $, we have
 \begin{eqnarray}
   h(x_1,x_2)-h'(x_1,x_2) &=& - f([x_1,x_2])+[x_1,f(x_2)]+[f(x_1),x_2]\nonumber\\
    &=& \delta^1(f)(X_1,x_2)
 \end{eqnarray}
Thus  two extensions $E_h$ and $E_{h'}$ are equivalent if and only if there exists some $f\in C_0^1(L;M)$ such that $\delta^1f = h-h'$. We thus have following theorem:

 \begin{thm}
   The set $F(L,M)$ of  all equivalence classes of  extensions of $L$ by $M$ is in one to one correspondence with the cohomology group $H_0^2(L;M)$. This correspondence $\omega :H_0^2(L;M)\to F(L,M)$ is obtained  by assigning to each cocycle $h\in Z_0^2(L;M)$, the extension given by multiplication \ref{SLCOH7}.
 \end{thm}

\section{ Deformation of  Leibniz Superalgebras }\label{rbsec5}
Let $L=L_0\oplus L_1$ be a Leibniz superalgebra. We denote the ring of all formal power series with coefficients in L by $L[[t]]$. Clearly,  $L[[t]]=L_0[[t]]\oplus L_1[[t]]$. So every $a_t\in L[[t]]$ is of the form $a_t=a_{t_0}\oplus a_{t_1}$, where $a_{t_0}\in L_0[[t]]$ and $a_{t_1}\in L_1[[t]]$.

\begin{defn}\label{rb2}
Let $L=L_0\oplus L_1$ be a Leibniz superalgebra.  A formal one-parameter deformation of a Leibniz superalgebra $L$ is a $K[[t]]$-bilinear map  $$\mu_t : L[[t]]\times L[[t]]\to L[[t]]$$ satisfying the following properties:
\begin{itemize}
  \item[(a)]  $\mu_t(a,b)=\sum_{i=0}^{\infty}\mu_i(a,b) t^i$, for all  $a,b\in L$, where $\mu_i:L\times L\to L$, $i\ge 0$ are  bilinear homogeneous mappings of degree zero  and $\mu_0(a,b)=[a,b]$ is the original  product on L.
\item[(b)]\begin{equation}\label{DLT1}
   \mu_t( \mu_t(a,b),c)=\mu_t(a,\mu_t(b,c))-(-1)^{ab}\mu_t(b,\mu_t(a,c)),
  \end{equation}
  for all homogeneous $a,b,c\in L$.
\end{itemize}
The Equation  \ref{DLT1} is equivalent to following equation:
   \begin{eqnarray}\label{rbeqn1}
      &&\sum_{i+j=r} \mu_i( \mu_j(a,b),c)\notag\\
      &=& \sum_{i+j=r}\{\mu_i(a,\mu_j(b,c))-(-1)^{ab}\mu_i(b,\mu_j(a,c))\},
   \end{eqnarray}
for all homogeneous $a,b,c\in L$.
\end{defn}
 Now we define a formal deformation of finite order of a Leibniz superalgebra $L$.
\begin{defn}\label{rb3}
Let L be a Leibniz superalgebra.  A  formal one-parameter deformation of order $n$  of  L is a $K[[t]]$-bilinear map  $$\mu_t : L[[t]]\times L[[t]]\to L[[t]]$$ satisfying the following properties:
\begin{itemize}
  \item[(a)]  $\mu_t(a,b)=\sum_{i=0}^{n}\mu_i(a,b) t^i$,  $\forall a,b,c\in L$, where $\mu_i:L\times L\to T$, $0\le i\le n$, are $K$-bilinear homogeneous maps of degree $0$, and $\mu_0(a,b)=[a,b]$ is the original product on $L$.
      \item[(b)] \begin{equation}\label{FDLS}
   \mu_t( \mu_t(a,b),c)=\mu_t(a,\mu_t(b,c))-(-1)^{ab}\mu_t(b,\mu_t(a,c)),
  \end{equation}
  for all homogeneous $a,b,c\in L$.
\end{itemize}
\end{defn}
\begin{rem}\label{rbrem1}
  \begin{itemize}
    \item For $r=0$, conditions \ref{rbeqn1} is equivalent to the fact that $L$ is a Leibniz superalgebra.
    \item For $r=1$, conditions \ref{rbeqn1} is equivalent to
     \begin{eqnarray*}\label{rrbeqn1}
      0&=&-\mu_1([a,b],c)-[\mu_1(a,b),c]+ \mu_1(a,[b,c])\\
      &&-(-1)^{ab}\mu_1(b,[a,c])+ [a,\mu_1(b,c)]-(-1)^{ab}[b,\mu_1(a,c)]\\
       &=&\delta^2\mu_1(a,b,c); \;\text{for all homogeneous}\; a,b,c\in L.
   \end{eqnarray*}
         %From (ii), (iv) and (vi), we have $\mu_1$, $\nu_1$ and $\phi$ are G-equivariant.
          Thus for $r=1$,  \ref{rbeqn1} is equivalent to saying that $\mu_1\in C_0^2(L;L)$  is a cocycle. In general, for $r\ge 0$, $\mu_r$ is just a 2-cochain, that is,  in $\mu_r\in C_0^2(L;L).$
  \end{itemize}
\end{rem}
\begin{exmpl}
Consider the Leibniz superalgebra $L=L_0\oplus L_1$ in Example \ref{rbLSe1}.  Define a bilinear  mapping $\mu_1:L\times L\to L$ by
$$\mu_1(z,z)=x,\;\mu_1(x,x)=\mu_1(x,y)=\mu_1(y,x)=\mu_1(y,y)=0,$$ $$\mu_1(z,x)=\mu_1(x,z)=\mu_1(y,z)=\mu_1(z,y)=0.$$ Clearly, $\mu_1$ is homogeneous of degree $0.$ One can easily verify that $\mu_t=\mu_0+\mu_1t$, where $\mu_0=[-,-]$ is the product in the Leibniz superalgebra $L$,  is a  formal one parameter deformation of $L.$
\end{exmpl}
\begin{defn}
  The cochain  $\mu_1 \in C_0^2(L;L)$ is called infinitesimal of the   deformation $\mu_t$. In general, if $\mu_i=0,$ for $1\le i\le n-1$, and $\mu_n$ is a nonzero cochain in  $C_0^2(L;L)$, then $\mu_n$ is called n-infinitesimal of the  deformation $\mu_t$.
\end{defn}
\begin{prop}
  The infinitesimal   $\mu_1\in C_0^2(L;L)$ of the  deformation  $\mu_t$ is a cocycle. In general, n-infinitesimal  $\mu_n$ is a cocycle in $C_0^2(L;L).$
\end{prop}
\begin{proof}
  For $n=1$, proof is obvious from the Remark \ref{rbrem1}. For $n>1$, proof is similar.
\end{proof}
\section{Equivalence of Formal  Deformations and Cohomology }\label{rbsec6}
Let   $\mu_t$  and $\tilde{\mu_t}$ be two formal deformations of a Leibniz superalgebra $L-L_0\oplus L_1$. A formal isomorphism from the deformation $\mu_t$ to $\tilde{\mu_t}$  is a $K[[t]]$-linear automorphism $\Psi_t:L[[t]]\to L[[t]]$ of the  form  $\Psi_t=\sum_{i=0}^{\infty}\psi_it^i$, where each $\psi_i$ is a homogeneous $K$-linear map $L\to L$ of degree $0$, $\psi_0(a)=a$, for all $a\in T$ and $$\tilde{\mu_t}(\Psi_t(a),\Psi_t(b))=\Psi_t\circ\mu_t(a,b),$$ for all $a,b\in L.$
\begin{defn}
  Two  deformations $\mu_t$  and $\tilde{\mu_t}$ of a Leibniz superalgebra $L$ are said to be equivalent if there exists a formal isomorphism  $\Psi_t$ from $\mu_t$ to  $\tilde{\mu_t}$.
\end{defn}
 Formal isomorphism on the collection of all  formal deformations of a Leibniz superalgebra $L$ is an equivalence relation.
\begin{defn}
  Any formal deformation of T that is equivalent to the deformation $\mu_0$ is said to be a trivial deformation.
\end{defn}

\begin{thm}
  The cohomology class of the infinitesimal of a  deformation $\mu_t$ of   a Leibniz Superalgebra $L$ is determined by the equivalence class of $\mu_t$.
\end{thm}
\begin{proof}
  Let  $\Psi_t$ be a formal  isomorphism  from  $\mu_t$ to  $\tilde{\mu_t}$. So  we have, for  all $a,b\in L$,  $\tilde{\mu_t}(\Psi_ta,\Psi_tb)=\Psi_t\circ \mu_t(a,b)$. This implies that \begin{eqnarray*}
     (\mu_1-\tilde{\mu_1})(a,b)&=& [\psi_1a,b]+[a,\psi_1b]-\psi_1([a,b])\\
     &=& \delta^1\psi_1(a,b).
  \end{eqnarray*}
 So we have $\mu_1-\tilde{\mu_1}=\delta^1\psi_1$. This completes the proof.
\end{proof}

%\section*{Acknowledgements}
%\section*{References}

%\bibliographystyle{alpha}
\bibliography{raj}

\end{document}